\documentclass[12pt]{amsart}

\usepackage{amssymb, enumerate, ulem}
\usepackage{hyperref}
\usepackage{geometry}
\usepackage[utf8]{inputenc}
\usepackage[T1]{fontenc}
\usepackage{textcomp}
\usepackage{mathtools}

\usepackage{enumerate}
\usepackage[toc,page]{appendix}
\usepackage[utf8]{inputenc}

\usepackage{graphicx,amssymb,amstext,amsmath,amsthm,amscd,amsfonts,hyperref, mathrsfs}
\usepackage{mathptmx}
\usepackage{wrapfig}
\usepackage{enumitem}
\usepackage{float}
\usepackage[usenames, dvipsnames]{color}
\usepackage{xcolor}
\usepackage{listings}
\usepackage{tikz,tkz-tab}
\usepackage{tikz-cd}
\usepackage{emptypage}

\usepackage{fancyhdr}

\newtheorem{theorem}{Theorem}[section]
\newtheorem{cor}[theorem]{Corollary}
\newtheorem{defn}[theorem]{Definition}
\newtheorem{exam}[theorem]{Example}
\newtheorem{lemma}[theorem]{Lemma}
\newtheorem{prop}[theorem]{Proposition}
\newtheorem{rem}[theorem]{Remark}
\title{Seshadri Constants  Over fields of characteristic zero }
\author{Shripad M. Garge}
\address{SMG: Department of Mathematics, Indian Institute of Technology Bombay, Powai, Mumbai. 400 076. INDIA.}
\email{shripad@math.iitb.ac.in, smgarge@gmail.com}
\author{Arghya Pramanik}
\address{AP: Department of Mathematics, Indian Institute of Technology Bombay, Powai, Mumbai. 400 076. INDIA.}
\email{arghya@math.iitb.ac.in, arghyapramanik4@gmail.com}

\begin{document}

\begin{abstract}
Let $X$ be a smooth projective variety defined over a field $k$ of characteristic $0$ and let $\mathcal{L}$ be a nef line bundle defined over $k$. 
We prove that if $x\in X$ is a $k$-rational point then the Seshadri constant $\epsilon(X, \mathcal{L}, x)$ over $\overline{k}$ is the same as that over $k$.
We show, by constructing families of examples, that there are varieties whose global Seshadri constant $\epsilon(X)$ is zero.
We also prove a result on the existence of a Seshadri curve with a natural (and necessary) hypothesis.
\end{abstract}

\maketitle

\section{Introduction}
Let $X$ be a smooth projective variety defined over an algebraically closed field and let $\mathcal{L}$ be a line bundle on $X$. 
Let us fix $x \in X$. 
For any curve $C$ in $X$ passing through $x$, the multiplicity of $C$ at $x$ is denoted by ${\rm{mult}}_xC$.
Seshadri ampleness criterion (\cite[Theorem 7.1]{Har}) tells us that $\mathcal{L}$ is ample if and only if for any irreducible curve $C$ such that $x \in C \subseteq X$ there exists a positive real number $\epsilon > 0$ such that 
$$\mathcal{L}\cdot C \geq \epsilon ~{\rm{mult}}_x C$$
where $\mathcal{L}\cdot C$ is the intersection pairing. 
Using this criterion, Demailly first defined Seshadri constant in \cite{Dem}, with the hope that it could be used to prove cases of Fujita's conjecture.
We give Demailly's definition below. 

\begin{defn}\cite[6.1]{Dem}\label{def1}
Let $X$ be as above, $x \in X$ and let $\mathcal{L}$ be a nef line bundle on $X$. 
The {\rm Seshadri constant} of $\mathcal{L}$ at $x$, denoted by $\epsilon(X,\mathcal{L},x)$, is defined by
$$\epsilon(X,\mathcal{L},x) := \inf_{x \in C} \frac{\mathcal{L}\cdot C}{\rm{mult}_xC}$$ 
where the infimum is taken over all irreducible curves $C$ in $X$ passing through $x$.
\end{defn}

We can now reformulate the Seshadri ampleness criterion as ``$\mathcal{L}$ is ample if and only if the Seshadri constant $\epsilon(X,\mathcal{L},x)$ is positive for every point $x \in X$''. 

The computations of Seshadri constants are considered to be rather hard in general. 
The known results are very specific, in the sense that the computations are done for specific varieties. 
In some cases, upper bounds on Seshadri constants are known. 
It is clear that $\epsilon(X,\mathcal{L},x) \ge 0$ but no example is known with $\epsilon(X, \mathcal{L})=\inf_{x \in X}\epsilon(X,\mathcal{L},x) =0$. 
Further, it is not known if $\epsilon(X,\mathcal{L},x)$ can be an irrational number. 
For more information on Seshadri constants, the motivations behind defining them and the known results, we refer the reader to the excellent survey on the topic, \cite{primer}.

In this paper, we begin with a smooth projective variety $X$ defined over a characteristic zero field $k$ and a nef line bundle $\mathcal{L}$ on $X$ defined over $k$. 
If $x \in X$ is a $k$-rational point then we prove that the Seshadri constant $\epsilon(X, \mathcal{L}, x)$ can be computed over $k$, that is, the infimum given in the above definition needs to be taken only over the curves which are defined over $k$ (Corollary \ref{Cor1}). 

To be able to work over a general field $k$, not necessarily algebraically closed, we need to use another definition of Seshadri constants. 

\begin{defn}[\cite{Bau}]\label{def2}
Let $X, \mathcal{L}$ and $x$ be as above. 
Let $\pi:\tilde{X}\to X$ be the blow-up of $X$ at $x$ with the exceptional divisor $E$. 
Then
$$\epsilon(X,\mathcal{L},x) := \sup\{\lambda \geq 0 : \pi^*(\mathcal{L})-\lambda\cdot E {\rm ~is~nef}\} .$$
\end{defn}

This definition is more useful than the earlier one because, for instance, one gets the following nice upper bound using this definition. 

\begin{prop}[\cite{Laz}]\label{Laz:bound}
Let $X$ be an $n$ dimensional projective variety defined over an algebraically closed field, $\mathcal{L}$ be a nef line bundle on $X$ and $x \in X$. 
Then $\epsilon(X,\mathcal{L},x) \leq \sqrt[n]{\mathcal{L}^n}$.
\end{prop}
	
Further, to be able to work over a non-algebraically closed field, we use scheme-theoretic language and results from \cite{Das}. 
We also need to carefully develop the intersection theory on varieties defined over a general field (\cite{Das, Liu}).
We then define the {\it local Seshadri constant}, the version of Seshadri constant relative to the base field $k$, and show that this version is well-behaved under base extensions. 
This means that, if $K/k$ is an extension of fields then the Seshadri constant at $y_0 \in X_{K}$ is the same as the Seshadri constant at $x_0 \in X$ where $X_K$ is the base extension of $X$ and $y_0$ maps to $x_0$ under the projection map (Theorem \ref{Main1}).

Further, in $\S 4$, we study the behaviour of Seshadri constants over finite \'{e}tale covers.
We show, by giving examples, that the Seshadri constants can be unbounded over a family of finite \'{e}tale covers.
This is generalisation of results proved in \cite{roth} where authors worked over algebraically closed fields. 

In the final section, we prove, by giving a family of examples, the existence of smooth projective varieties over a characteristics zero field whose global Seshadri constant, $\epsilon(X)$, is zero (Corollary \ref{cor-Seshadri-curves}).
We close the paper with the existence of Seshadri curves (Theorem \ref{Main2}) under conditions which are similar to \cite[Lemma 5.2]{Bau} where, again, the author was working over algebraically closed fields. 

\section*{Acknowledgements}
The authors would like to thank Dipendra Prasad, Sudarshan Gurjar and Saurav Bhaumik for their inputs at various stages of this work. 
They would also like to thank Indranil Biswas for his interest in this work and Omprokash Das for his permission to include some of the results (and proofs) from \cite{Das}.

This question was inspired by some discussions with Indranil Biswas and Krishna Hanumanthu during the ICTS meeting on ``Topics in Birational Geometry''. 
The first named author would like to thank the organisers of the meeting, ICTS/Prog-TBG2020/01, for inviting him to the meeting and to the ICTS for local hospitality. 

\section{Seshadri constants on varieties over arbitrary fields}

Let $k$ be a field of characteristic zero. 

A {\it variety} over $k$ is an integral scheme which is separated and of finite type over $k$. 
A {\it subvariety} $Y$ of $X$ is a closed subscheme of $X$ which is a variety over the ground field of $X$.
In this paper, we assume that all varieties are normal, geometrically integral and projective over the base field.

Due to the lack of a widely accepted reference on intersection theory over a general field, we give it in some detail here. 

Let $C$ be a curve over $k$. 
We express a Cartier divisor $W$ on $C$ by $\sum_{x \in C} ({\rm mult}_xW)x$, where the sum is taken over closed points. 
Following Liu (\cite{Liu}), we define the degree of $W$, $\deg_kW=\sum_x ({\rm mult}_x W) [k(x):k]$, where the sum is, again, taken over closed points. 

Now, let $X$ be a variety over $k$. 
Let $C$ be a curve in $X$ and $W$ be a Cartier divisor on $X$. 
Then we define, following Das and Waldron (\cite[section 3]{Das}), the intersection number of $C$ and $W$ over $k$ by $\deg_k W|_C$ and denote it by $W\cdot_kC$.

Let $X$ be a variety of dimension $n$ over $k$ and let $x \in X$ be a closed point. 
Let $\pi:\tilde{X}\to X$ be the blow-up of $X$ at $x$ with the exceptional divisor $E= {\rm Proj} (\oplus_{i\geq0}{m_x}^{i}/{m_x}^{i+1})$. 
The {\it $k$-multiplicity of $X$ at $x$} is defined by ${\rm mult}_{x/k}X = (-1)^{n+1}(E^n)_k$ where $(E^n)_k$ denotes the $n$-fold self intersection of $E$ over $k$ (\cite[Definition 9.1]{Das}).
It follows from \cite[Lemma 9.3]{Das} that if $X$ is smooth and if $x \in X$ is a closed point then ${\rm mult}_{x/k}X = [k(x):k]$.
In particular, for a $k$-rational point, $x \in X(k)$, the multiplicity of $X$ at $x$ is equal to $1$. 

Now we mention some results on the behaviour of blow-ups on a variety over an arbitrary field.

\begin{lemma}\cite[Lemma 9.4]{Das}\label{Das:9.4}
Let $X$ be a smooth, projective variety over $k$ with dimension at least 2 and let $W$ be an effective Cartier divisor on $X$. 
Let $h$ be a local equation of $W$ and let $x \in W$ be a closed point of $X$. 
Let $l\geq0$ be the largest integer such that $h\in m_x^l$.
Then we have 
$${\rm mult}_{x/k}W=l[k(x):k] \hskip5mm {\rm and} \hskip5mm \pi^*W=\tilde{W}+lE$$ 
where $m_x$ is the maximal ideal of the local ring $O_{X,x}$, $\pi:\tilde{X}\to X$ is the blow-up of $X$ at $x$ with exceptional divisor $E$ and $\tilde{W}$ is the strict transform of $W$ on $\tilde{X}$.
\end{lemma}

\begin{rem}\cite[Remark 9.5]{Das}\label{Das:9.5}
In the above setup, let $C$ be a curve on $X$ passing through $x$ and let $\tilde{C}$ be the strict transform of $C$ on $\tilde{X}$ then ${\rm mult}_{x/k}C=E\cdot_k\tilde{C}$.
\end{rem}

\begin{lemma}\cite[Chapter 8, Prop 1.12]{Liu}\label{Liu:8.1.12}
Let $X$ be a locally Noetherian scheme and let $\mathcal{I}$ be a quasi-coherent sheaf of ideals on $X$ and let $\pi:\tilde{X}\to X$ be the blow-up along the center $V(\mathcal{I})$. 
Let $Z\to X$ be a flat morphism of schemes where $Z$ is a locally Noetherian scheme. 
Let $\tilde{Z}\to Z$ be the blow-up of $Z$ along the center $\mathcal{I}O_Z$ then $\tilde{Z}\simeq\tilde{X}\times_X Z$.
\end{lemma}

We have the following version of Bezout's Theorem in the above setup.

\begin{lemma}\cite[Corollary 9.6]{Das}\label{Das:9.6}
Let $X$ be a normal projective variety over $k$ and $x\in X$ is a regular closed point. 
Let $D$ be an effective cartier divisor passing through $x$ and $C$ be a curve on $X$ passing through $x$, such that $C$ is not in support of $D$. 
Then
$$D\cdot_kC\geq\frac{1}{[k(x):k]}({\rm mult}_{x/k}D)({\rm mult}_{x/k}C) .$$
\end{lemma} 

Now we state the Riemann Roch theorem for surfaces over an arbitrary field.

\begin{theorem}[\cite{Tan}, Theorem 2.10]\label{Tan:2.10}
Let $X$ be a smooth projective surface over a field $k$. 
Let $D$ be a $\mathbb{Z}$ divisor on $X$. 
Then $\chi(X,D)=\chi(X,O_X)+\frac{1}{2}D\cdot_k(D-K_X)$.
\end{theorem}

We now define Seshadri constants on varieties over arbitrary fields using intersection theory developed above.
We see that, in analogy with the Definitions \ref{def1} and \ref{def2} over algebraically closed fields, we have an equivalent formulation of local Seshadri constants over arbitrary fields.

\begin{defn}\cite[Definition 9.7]{Das}\label{Das:9.7}
Let $X$ be a smooth projective variety over a field $k$ and let $x \in X$ be a closed point. 
Let $\pi:\tilde{X}\to X$ be the blow-up of $X$ at $x$ with exceptional divisor $E$ and let $\mathcal{L}$ be a nef line bundle on $X$. 
We define the Seshadri constant at $x$ denoted by $\epsilon(X,\mathcal{L},x)$ as
$$\epsilon(X,\mathcal{L},x):=\sup\{\epsilon\geq0:\pi^*\mathcal{L}-\epsilon E\text{ is nef on } \tilde{X}\}.$$
\end{defn}

We also define the multi-point Seshadri constant in the similar way.

\begin{defn}\cite[Definition 9.7]{Das}\label{Das:multi:9.7}
Let $X$ be a smooth projective variety over a field $k$ and let $x_1,x_2\cdots,x_t \in X$ be closed points. 
Let $\pi:\tilde{X}\to X$ be the blow-up of $X$ at $\{x_1,x_2\cdots,x_t \}$  with exceptional divisor $E$ and let $\mathcal{L}$ be a nef line bundle on $X$. 
We define the multipoint Seshadri constant at $x_1, x_2\cdots,  x_t$ denoted by $\epsilon(X,\mathcal{L}, x_1,x_2\cdots, x_t)$ as
$$\epsilon(X,\mathcal{L}, x_1,x_2\cdots, x_t):=\sup\{\epsilon\geq0:\pi^*\mathcal{L}-\epsilon E\text{ is nef on } \tilde{X}\}.$$
\end{defn}

In the above setup, we have the following equivalent definition of the Seshadri constant at a closed point $x\in X$. 
For the sake of completeness we reproduce the proof from \cite[Section 9]{Das}.

\begin{prop}\cite[Lemma 9.8]{Das}\label{Das:9.8}
We have $\epsilon(X,\mathcal{L},x)=inf\{\frac{\mathcal{L}\cdot_kC}{{\rm mult}_{x/k}C}\}$ where the infimum is taken over all integral curve $C$ in $X$ passing through $x$.
\end{prop}

\begin{proof}
We observe that to check the nef property of $\pi^*\mathcal{L}-\epsilon E$ on $\tilde{X}$ it is enough to check it for all strict transforms of the curves in $X$. 
Now it follows that $\pi^*\mathcal{L}-\epsilon E$ is nef on $\tilde{X}$ if and only if $(\pi^*\mathcal{L}-\epsilon E)\cdot_k\tilde{C}\geq0$ where $\tilde{C}$ denotes the strict transform of an integral curve $C$ in $X$ passing through $x$ on $\tilde{X}$. 
Now, from Remark \ref{Das:9.5}, we have $\pi^*\mathcal{L}\cdot_k\tilde{C}\geq\epsilon \;{\rm mult}_{x/k}C$.
Hence $\frac{\mathcal{L}\cdot_kC}{{\rm mult}_{x/k}C}\geq\epsilon$ if and only if $\pi^*\mathcal{L}-\epsilon E$ is nef on $\tilde{X}$.
\end{proof}

\begin{defn}
Let $X$ be a smooth projective variety over $k$ and $\mathcal{L}$ is a nef line $k$-bundle on $X$. Then we define following three kinds of Seshadri constants:

\noindent {\rm (i)} The {\rm Seshadri constant of the line bundle $\mathcal{L}$}:
$$\epsilon(X,\mathcal{L}):=inf_{x\in X}\epsilon(X,\mathcal{L}, x) ,$$
\noindent {\rm (ii)} The {\rm Seshadri constant at the point $x$}:
$$\epsilon(X, x):=inf_{\mathcal{L} \text{ is ample} }\epsilon(X,\mathcal{L}, x) ,$$
\noindent {\rm (iii)} The {\rm global Seshadri constant of the variety $X$}:
$$\epsilon(X):=inf_{\mathcal{L} \text{ is ample}}\epsilon(X,\mathcal{L})=inf_{x\in X}\epsilon(X, x) .$$

\end{defn}

	
\section{Seshadri constants and base extensions}
Now we consider the following setup.

Let $X$ be a smooth projective variety defined over a characteristic zero field $k$ with $H^0(X,O_X)=k$ and let $K$ be an extension of $k$. 
We consider the following fibered diagram:

\begin{center}
\begin{tikzcd}[row sep=huge, column sep=huge]
X_K\arrow[r, "h"] \arrow[d, "f"'] & {\rm spec}(K) \arrow[d] \\
X\arrow[r, "g"'] & {\rm spec}(k)
\end{tikzcd}
\end{center}
where $f$ is faithfully flat being the base change of a faithfully flat map, the natural map from spec$(K)$ to spec$(k)$. 
Let $x_0 \in X$ be a $k$ ratinal point. 
We have $f^{-1}(x_0)=y_0$ for some unique closed point $y_0\in X_K$.

\begin{lemma}\label{lem3.2}
With the above notations, let $(O_{X,x_0},m_{x_0})$ and $(O_{X_K,y_0},m_{y_0})$ denote the local rings at $x_0$ and $y_0$ respectively.
If $f':O_{X,x_0}\to O_{X_K,y_0}$ denotes the map induced by $f$ on the local rings then $m_{x_0}O_{X_K,y_0}=m_{y_0}$.
\end{lemma}

\begin{proof}
This follows because $f$ is an unramified morphism being a base extension of an unramified morphism.
\end{proof}

\begin{lemma}[\cite{Donu}, Lemma C.5]\label{Donu:C.5}
With $X$ and $K$ as above, a line bundle $\mathcal{L}$ on $X$ is nef if and only if the line bundle $\mathcal{L}\otimes_kK$ on $X\times_kK$ is nef.
\end{lemma}

\begin{theorem}\label{Main1}
Let $\mathcal{L}$ be a nef line bundle on $X$. 
Then $\epsilon(X_K,f^*\mathcal{L},y_0)=\epsilon(X,\mathcal{L},x_0)$.
\end{theorem}

\begin{proof}
First we notice from the above lemma that $f^*\mathcal{L}$ is a nef line bundle on $X_K$ so $\epsilon(X_K,f^*\mathcal{L},y_0)$ is well defined.
We take $\mathcal{I}$ to be the ideal sheaf corresponding to the inclusion $\{x_0\}\xhookrightarrow{}X$ then applying  Lemma \ref{lem3.2} and Lemma \ref{Liu:8.1.12} we have the following fibered diagram:
$$\begin{tikzcd}[row sep=huge, column sep=huge]
\tilde{X_K}\arrow[r, "f_{K}"] \arrow[d, "\pi_{K}"'] &\tilde{X} \arrow[d, "\pi"] \\
X_{K}\arrow[r, "f"'] & X
\end{tikzcd}$$
where $\pi$ is the blowup of $X$ at $x_0$ and $\pi_K$ is the blow up of $X_{K}$ at $y_0$. 
We then have the following commutative diagram of Picard groups.
$$\begin{tikzcd}[row sep=huge, column sep=huge]
Pic(X)\arrow[r, "f^*"] \arrow[d, "\pi^*"'] &Pic(X_K) \arrow[d, "\pi^*_K"] \\
Pic(\tilde{X})\arrow[r, "f_K^*"'] & Pic(\tilde{X_K})
\end{tikzcd}$$
Now let $\lambda\geq 0$ be a real number such that $\pi_K^*f^*\mathcal{L}-\lambda f_K^*E$ is nef on $X_K$. 

Using commutativity of the diagram in Picard groups we observe that $\pi_K^*f^*\mathcal{L}-\lambda f_{K}^*E$ is nef on $X_{K}$ if and only if $f_{K}^*\pi^*\mathcal{L}-\lambda f_{K}^*E$ is nef on $X_K$. 
We have $f_{K}^*\pi^*\mathcal{L}-\lambda f_{K}^*E$ is nef on $X_K$ if and only if $f_{K}^*(\pi^*\mathcal{L}-\lambda E)$ is nef on $X_{K}$. 

Now using Lemma \ref{Donu:C.5} we get $f_{K}^*(\pi^*\mathcal{L}-\lambda E)$ is nef on $X_{K}$ if and only if $\pi^*\mathcal{L}-\lambda E$ is nef on $X$. 
Hence we get $\epsilon(X_{K},f^*\mathcal{L},y_0)=\epsilon(X,\mathcal{L},x_0)$.
\end{proof}

\begin{cor}\label{Cor1}
Let $X$ be a smooth projective variety defined over $k$ the Seshadri constant at a rational point can be computed over curves defined over $k$ passing through that point.
\end{cor}

\begin{proof}
Proposition \ref{Das:9.8} and Theorem \ref{Main1} are equivalent formulations of Seshadri constants. 
Hence the result follows.
\end{proof}


\section{Seshadri constants of finite etale covers}
In this section, we see relations between Seshadri constants of finite etale covers and those of the base space. 
Similar results have been proved in \cite{roth} but only over algebraically closed fields. 
We consider Seshadri constants and multi-point Seshadri constants over characteristics zero fields, not necessarily algebraically closed.

\begin{theorem}\label{Main3}
Let $Y$ and $Z$ be smooth projective varieties over a field $k$ and let $g:Y\to Z$ be a finite \'{e}tale morphism over the field $k$. 
Let $\mathcal{L}$ be a nef line bundle on $Z$ and $g^{-1}(z)=\{y_1, y_2,\cdots, y_r\}$ for some $z\in Z$ and $y_1, y_2,\cdots, y_r$ are distinct points in $Y$ then 
$$\epsilon(Y,g^*\mathcal{L},y_1, y_2,\cdots, y_r)=\epsilon(Z,\mathcal{L},z) .$$
\end{theorem}

\begin{proof}
As $g$ is a proper morphism, the pullback of nef line bundle is nef so $g^*\mathcal{L}$ is a nef line bundle on $Y$. 
Also $g$ is flat and unramified, hence by Lemma \ref{Liu:8.1.12} we have the following fibered diagram:
$$\begin{tikzcd}[row sep=huge, column sep=huge]
\tilde{Y}\arrow[r, "h"] \arrow[d, "\pi_1"'] &\tilde{Z} \arrow[d, "\pi_2"] \\
Y\arrow[r, "g"'] & Z
\end{tikzcd}$$
where $\pi_1$ and $\pi_2$ denote blow-ups of $Y$ and $Z$ at $\{y_1, y_2,\cdots, y_r\}$ and $z$ respectively. 
	
We then get the following commutating diagram of Picard groups.
$$\begin{tikzcd}[row sep=huge, column sep=huge]
Pic(Z)\arrow[r, "g^*"] \arrow[d, "\pi_2^*"'] &Pic(Y) \arrow[d, "\pi_1^*"] \\
Pic(\tilde{Z})\arrow[r, "h^*"'] & Pic(\tilde{Y})
\end{tikzcd}$$
	
By Definition \ref{Das:multi:9.7}, 
$$\epsilon(Y,g^*\mathcal{L},y_1, y_2,\cdots, y_r)=\max\{\epsilon\geq0:\pi_1^*g^*\mathcal{L}-\epsilon (E_1+E_2\cdots +E_r)~\text{is nef}\}$$
where $E_1+E_2\cdots +E_r$ is the exceptional divisor of the blow-up $\pi_1$. 
So we have 
$$\epsilon(Y,g^*\mathcal{L},y_1, y_2,\cdots, y_r)=\max\{\epsilon\geq0:h^*\pi_2^*\mathcal{L}-\epsilon h^*F~\text{is nef}\}$$ 
where $F$ is the exceptional divisor of the blow-up $\pi_2$. 
Now we notice that $h$ is finite being base change of finite map.
Hence we have  $\epsilon(Y,g^*\mathcal{L},y_1, y_2,\cdots, y_r)=max\{\epsilon\geq0:\pi_2^*\mathcal{L}-\epsilon F\text{ is nef}\}$.
This proves the result.
\end{proof}

\begin{cor}
Let $Y$, $Z$, $g$, $\mathcal{L}$ and $r$ be as above. 
Then $\epsilon(Z,\mathcal{L},z)\leq \sqrt[m]{\frac{(g^*\mathcal{L}^m)_k}{d_1+d_2\cdots+d_r}}$ where $m>1$ is the dimension of $Y$ and $d_i=deg(y_i)$ for $i=1,2\cdots,r$.
\end{cor}

\begin{proof}
Let $\epsilon=\epsilon(Y,g^*\mathcal{L},y_1, y_2,\cdots, y_r)$ then we have $(\pi^*g^*\mathcal{L}-\epsilon (E_1+E_2\cdots +E_r))^m_k\geq 0$ where $E_1, E_2\cdots E_r$ are as above. Now $(\pi^*g^*\mathcal{L}-\epsilon (E_1+E_2\cdots +E_r))^m_k=(g^*\mathcal{L}^m)_k-(d_1+d_2\cdots+d_r)\epsilon^m.$
Hence $\epsilon(Y,g^*\mathcal{L},y_1, y_2,\cdots, y_r)\leq \sqrt[m]{\frac{(g^*\mathcal{L}^m)_k}{d_1+d_2\cdots+d_r}}$. 
Hence the result follows from Theorem \ref{Main3}.
\end{proof}

\begin{cor}\label{Cor4.4}
Let $Y, Z, g, \mathcal{L}$ and $y_1$ be as above. 
Then $\epsilon(Y,g^*\mathcal{L},y_1)\geq\epsilon(Z,\mathcal{L},z)$.
\end{cor}

\begin{proof}
We have $h^*F=E_1+E_2 +\cdots +E_r$ where $h$, $F$ and $E_1,E_2,\cdots ,E_r$ are as in Theorem \ref{Main3}. 
We have that $\epsilon(Y,g^*\mathcal{L},y_1) = \max\{\epsilon\geq0:\pi_1^*g^*\mathcal{L}-\epsilon E_1$ is nef$\}$ which is the same as $\max\{\epsilon\geq0:\pi_1^*g^*\mathcal{L}-\epsilon (h^*F-(E_2+\cdots+E_r))$ is nef$\}$. 
Using commutativity of diagram of Picard groups in the proof of Theorem \ref{Main3} we get that $\epsilon(Y,g^*\mathcal{L},y_1)=\max\{\epsilon\geq0:h^*\pi_2^*\mathcal{L}-\epsilon (h^*F-(E_2+\cdots+E_r))$ is nef$\}$. 
Therefore, $\epsilon(Y,g^*\mathcal{L},y_1)\geq \max\{\epsilon\geq0:h^*\pi_2^*\mathcal{L}-\epsilon h^*F$ is nef$\}$. 
Now as $h$ is finite morphism, $\epsilon(Y,g^*\mathcal{L},y_1)\geq \max\{\epsilon\geq0:\pi_2^*\mathcal{L}-\epsilon F$ is nef$\}$ and hence the result.
\end{proof}

If the point $x \in X$ is a closed point which is not a $k$-rational point then it is not necessary that the Seshadri constants remain the same over $k$ and over $k(x)$.
The following proposition gives a lower bound for the Seshadri constant over $k(x)$.

\begin{prop}
Let $X$ be a smooth projective variety defined over a characteristic zero field $k$ and $x_0$ be a closed point of $X$. 
Let $Q$ denote the finite extension $k(x_0)$ of $k$ and let $K$ be a field extension of $k$ containing $Q$. 
Let $f:X_K \to X$ denote the natural map and $y_0\in f^{-1}\{x_0\}$.
Then
$$\epsilon(X_K,f^*\mathcal{L},y_0)\geq\epsilon(X,\mathcal{L},x_0) .$$
\end{prop}

\begin{proof}
We consider the following commutative diagram
$$
\begin{tikzcd}[row sep=huge, column sep=huge]
X_{K}\arrow[r,"e"] \arrow[d]\arrow[rr, "f", bend left=40] &X_{Q}\arrow[r, "a"] \arrow[d] &X \arrow[d] \\
spec(K) \arrow[r] &spec(Q)\arrow[r, "c"'] & spec(k)
\end{tikzcd}
$$
From the diagram, we see that $a:X_{Q}\to X$ is a finite map.
Further, $x_0$ is $Q$-rational point of $X_Q$ and $y_0$ is the unique point in $X_{K}$ such that $e(y_0)=x_0$. 
Corollary \ref{Cor4.4} then gives $\epsilon(X_Q,a^*\mathcal{L},x_0)\geq\epsilon(X,\mathcal{L},x_0)$. 
Further, by Theorem \ref{Main1}, we have 
$$\epsilon(X_K, f^*(\mathcal{L}), y_0) = \epsilon(X_K, e^*a^*(\mathcal{L}), y_0)=\epsilon(X_Q,a^*\mathcal{L},x_0) .$$
This proves the required inequality.
\end{proof}

In the following example given in \cite{roth}  we see that Seshadri constants of finite etale cover can be very large.

\begin{theorem}\label{thm4.6}
Let $X$ be an abelian variety and $\mathcal{L}$ be an ample line bundle on $X$. 
Let $[n]:X\to X$ is the multiplication by $n$ map where $n$ is a positive integer. 
Then	
$$[n]^*\mathcal{L}\simeq{\mathcal{L}^{\otimes\frac{n^2+n}{2}}\otimes([-1]^*\mathcal{L})^{\otimes\frac{n^2-n}{2}}} .$$
\end{theorem}

\begin{exam}{\rm 
Let $X$ be an abelian variety over a characteristic zero field $k$ and $L$ be an ample line bundle on $X$. 
Assume that the group of $k$-rational points on $X$ is torsion-free. 
For $n \in \mathbb{N}$, let $[n]:X\to X$ be the map $x\to n.x$. 
Now, let us denote by $\mathcal{L}$ the line bundle $L \otimes[-1]^* L$. 

We note that $[n]$ is a finite \'{e}tale cover 
and we note that $[-1]^*\mathcal{L}=\mathcal{L}$.
Then using \ref{thm4.6}  we see that  $[n]^*\mathcal{L}$ is numerically equivalent to $n^2\mathcal{L}$ so for any rational point $y\in X$ we have $\epsilon(X,[n]^*\mathcal{L},y)=n^2\epsilon(X,\mathcal{L},y)$.

This shows that the Seshadri constants of finite \'{e}tale covers can be unbounded.
}\end{exam}

\section{An application, existence of a Seshadri curve}

Let $X$ be a projective variety defined over a field $k$ and let $\epsilon(X, \mathcal{L}, x)$ be the Seshadri constant of a line bundle $\mathcal{L}$ at a point $x$ in $X$.
As we know, this is the infimum of certain fractions over curves passing through $x$. 
If $X$ contains a curve $C$ over $k$ passing through $x$ such that $\epsilon(X, \mathcal{L}, x) = \frac{\mathcal{L}\cdot_k C}{{\rm mult}_{x/k}C}$ then the curve $C$ is called a {\it Seshadri curve} for $\mathcal{L}$ and $x$. 

It is an important problem to establish the existence of a Seshadri curve and is related to computing canonical slope of a line bundle (\cite[Section 5]{Bau}).
We prove the existence of a Seshadri curve under some natural conditions which are analogous to \cite[Lemma 5.2]{Bau}, which is stated for an algebraically closed field. 
We recover the statement when $\alpha$, as defined below, is equal to $1$.

We first give an upper bound of Seshadri constants over arbitrary fields.

\begin{prop}\label{Seshadri-curves}
Let $X$ be a smooth projective surface over a field $k$ and $x\in X$ be a closed point. 
Let $\alpha=[k(x):k]$.
Then for any nef line bundle $\mathcal{L}$ on $X$ over $k$ 
$$\epsilon(X, \mathcal{L}, x)\leq\frac{1}{\sqrt{\alpha}}\sqrt{{\mathcal{L}_k}^2} .$$
\end{prop}

\begin{proof}
We fix a nef line bundle $\mathcal{L}$ on $X$ over $k$.
Let $\pi: \tilde{X}\to X$ be the blow-up of $X$ at $x$ and let us denote by $\epsilon$ the Seshadri constant $\epsilon(X, \mathcal{L}, x)$. 
Then from Definition \ref{Das:9.7} we have $(\pi^*(\mathcal{L})-\epsilon E)_k^2\geq0$. 
Now since $E_k^2=-\alpha$ we have ${\mathcal{L}_k}^2-\alpha{\epsilon}^2\geq0$ we have the result.
\end{proof}

The above proposition enables us to prove a result about the global Seshadri constant of a projective variety $X$. 

\begin{cor}\label{cor-Seshadri-curves}
Let $X$ be a smooth projective surface defined over a field $k$.
Assume further that for every positive integer $n$ there is a closed point $x_n \in X$ such that $[k(x_n): k] > n$. 
In other words, we assume that the degrees of extensions $k(x)$ of closed points $x \in X$ are unbounded. 

Then the global Seshadri constant of $X$ is zero.
\end{cor}

This corollary gives us examples of varieties whose global Seshadri constants are zero. 
The hypothesis in the above corollary is satisfied by, the projective space $\mathbb{P}^2$, for instance. 
The hypothesis will be satisfied by any surface whose $\overline{k}$-points are not achieved over any finite extension of $k$. 

Note that, if $k$ is an algebraically closed field then it is expected that the global Seshadri constant of a projective $k$-variety $X$ is zero. 
Our corollary above is consistent with this. 

Now we give our main result regarding existence of Seshadri curves over a  characteristics zero field.

\begin{theorem}\label{Main2}
Let $X$ be a smooth projective normal surface over a characteristics zero field $k$.
Let $x \in X$ be a closed point with $\alpha=[k(x):k]$ and $\mathcal{L}$ be an ample line bundle on $X$ such that $\epsilon(X,\mathcal{L},x)<\frac{1}{\sqrt{\alpha}}\sqrt{{\mathcal{L}}^2_k}$. 
Then there exists a curve $C$ in $X$ such that $\epsilon(X,\mathcal{L},x)=\frac{\mathcal{L}\cdot_{k}C}{{\rm mult}_{x/k}C}$.
\end{theorem}

\begin{proof}
Let $\gamma$ be a rational number such that $\epsilon(X,\mathcal{L},x)<\gamma<\frac{1}{\sqrt{\alpha}}\sqrt{{\mathcal{L}}^2_k}$. 
From Proposition \ref{Das:9.8} we have a sequence of  irreducible curves $C_n$ passing through $x$ in $X$ such that $\frac{\mathcal{L}\cdot_{k}C_n}{{\rm mult}_{x/k}C_n}<\gamma$ for all $n\geq N$. 
We choose an integer $r$ such that $r\mathcal{L}-K_{X}$ is ample where $K_X$ is the canonical bundle.

\noindent Claim: There exists an integer $d$ and a curve $D\in |d\mathcal{L}|$ passing through $x$ such that 
$$\frac{\mathcal{L}\cdot_{k}D}{{\rm mult}_{x/k}D}\leq \frac{\mathcal{L}^2_k}{\alpha\gamma} .$$

We choose $d>r$ such that $m=d\gamma$ is an integer. 
By Serre vanishing theorem we have $h^1(d\mathcal{L})=h^2(d\mathcal{L})=0$ for large enough $d$. 
Now using Theorem \ref{Tan:2.10} we have 
$$h^0(d\mathcal{L})=\frac{d\mathcal{L}\cdot_{k}(d\mathcal{L}-K_X)}{2}+\chi(O_X) =\frac{d(d-r)\mathcal{L}^2_k}{2}+\frac{d\mathcal{L}\cdot_k(r\mathcal{L}-K_X)}{2}+\chi(O_X) .$$
Since $r\mathcal{L}-K_{X}$ is ample, we get $h^0(d\mathcal{L})>\frac{d(d-r)\mathcal{L}^2_k}{2}+\chi(O_X)$.

Now we prove that $|d\mathcal{L}|$ contains a curve $D$ with ${\rm mult}_xD \geq \alpha m$ if $h^0(d\mathcal{L})>\frac{\alpha m(m+1)}{2}$.

Since the curves $C$ passing through $x$ in the linear system are in a bijective correspondence with the global sections of the line bundle, we have 
$${\rm mult}_{x/k}C=\alpha\sup\{r|s_x=0 ,s_x\in m_x^r\mathcal{L}_x\}$$ 
where $s$ is the section corresponding to $C$.
Now consider the evaluation map $\phi:|d\mathcal{L}| \to \frac{O_{X,x}}{{m_x}^m}$. 
The dimension of $O_{X,x}/{m_x}^m$ is $\alpha m(m+1)/2$ which is less than the dimension of $|d\mathcal{L}|$. 
Therefore the kernel of $\phi$ is non empty.
So $|d\mathcal{L}|$ contains a curve $D$ with ${\rm mult}_xD\geq\alpha.m$ if  $\frac{d(d-r)\mathcal{L}^2_k}{2}+\chi(O_X)>\frac{\alpha.m(m+1)}{2}$, that is, $(\mathcal{L}_k^2-\alpha{\gamma}^2)d^2-(\alpha\gamma+r\mathcal{L}_k^2)d+2\chi(O_X)>0$.
This is a qudratic in $d$ with positive leading term. 
If we choose $d$ large enough such that the above term is positive then $|d\mathcal{L}|$ contains a curve $D$ with ${\rm mult}_xD\geq\alpha m$ for some large enough $d$.

Hence $\frac{\mathcal{L}\cdot_kD}{{\rm mult}_{x/k}D}\leq\frac{\mathcal{L}\cdot_kD}{\alpha m}=\frac{\mathcal{L}^2_k}{\gamma \alpha}$.

We now claim that every irreducible and reduced curve $C$ in $X$ satisfying $\frac{\mathcal{L}\cdot_kC}{{\rm mult}_{x/k}C}<\gamma$ is a component of $D$.
If $C$ is not a component of $D$ then using Lemma \ref{Das:9.6} we have 
$$C\cdot_kD = d\mathcal{L}\cdot_kC\geq \frac{1}{\alpha}({\rm mult}_{x/k}D)({\rm mult}_{x/k}C) >C\cdot_kD$$ 
which is a contradiction.
Hence the result.
\end{proof}

\end{document}